\documentclass[12pt]{amsart}
\usepackage{amsmath,amssymb,amsbsy,amsfonts,latexsym,amsopn,amstext,
                                               amsxtra,euscript,amscd,bm}
\usepackage{url}
\usepackage[colorlinks,linkcolor=blue,anchorcolor=blue,citecolor=blue]{hyperref}
\usepackage{color}
\usepackage{float}

\hypersetup{breaklinks=true}

\begin{document}

\newtheorem{theorem}{Theorem}
\newtheorem{thm}[theorem]{Theorem}
\newtheorem{lem}[theorem]{Lemma}
\newtheorem{claim}[theorem]{Claim}
\newtheorem{cor}[theorem]{Corollary}
\newtheorem{prop}[theorem]{Proposition}
\newtheorem{definition}{Definition}
\newtheorem{question}[theorem]{Question}
\newcommand{\hh}{{{\mathrm h}}}

\numberwithin{equation}{section}
\numberwithin{theorem}{section}
\numberwithin{table}{section}

\def\sssum{\mathop{\sum\!\sum\!\sum}}
\def\ssum{\mathop{\sum\ldots \sum}}
\def\iint{\mathop{\int\ldots \int}}

\def\squareforqed{\hbox{\rlap{$\sqcap$}$\sqcup$}}
\def\qed{\ifmmode\squareforqed\else{\unskip\nobreak\hfil
\penalty50\hskip1em\null\nobreak\hfil\squareforqed
\parfillskip=0pt\finalhyphendemerits=0\endgraf}\fi}

\newfont{\teneufm}{eufm10}
\newfont{\seveneufm}{eufm7}
\newfont{\fiveeufm}{eufm5}
%
%
\newfam\eufmfam
     \textfont\eufmfam=\teneufm
\scriptfont\eufmfam=\seveneufm
     \scriptscriptfont\eufmfam=\fiveeufm
%
%
\def\frak#1{{\fam\eufmfam\relax#1}}

\newcommand{\bflambda}{{\boldsymbol{\lambda}}}
\newcommand{\bfmu}{{\boldsymbol{\mu}}}
\newcommand{\bfxi}{{\boldsymbol{\xi}}}
\newcommand{\bfrho}{{\boldsymbol{\rho}}}

\def\fK{\mathfrak K}
\def\fT{\mathfrak{T}}

\def\fA{{\mathfrak A}}
\def\fB{{\mathfrak B}}
\def\fC{{\mathfrak C}}

\def \balpha{\bm{\alpha}}
\def \bbeta{\bm{\beta}}
\def \bgamma{\bm{\gamma}}
\def \blambda{\bm{\lambda}}
\def \bchi{\bm{\chi}}
\def \bphi{\bm{\varphi}}
\def \bpsi{\bm{\psi}}

\def\eqref#1{(\ref{#1})}

\def\vec#1{\mathbf{#1}}


\def\cA{{\mathcal A}}
\def\cB{{\mathcal B}}
\def\cC{{\mathcal C}}
\def\cD{{\mathcal D}}
\def\cE{{\mathcal E}}
\def\cF{{\mathcal F}}
\def\cG{{\mathcal G}}
\def\cH{{\mathcal H}}
\def\cI{{\mathcal I}}
\def\cJ{{\mathcal J}}
\def\cK{{\mathcal K}}
\def\cL{{\mathcal L}}
\def\cM{{\mathcal M}}
\def\cN{{\mathcal N}}
\def\cO{{\mathcal O}}
\def\cP{{\mathcal P}}
\def\cQ{{\mathcal Q}}
\def\cR{{\mathcal R}}
\def\cS{{\mathcal S}}
\def\cT{{\mathcal T}}
\def\cU{{\mathcal U}}
\def\cV{{\mathcal V}}
\def\cW{{\mathcal W}}
\def\cX{{\mathcal X}}
\def\cY{{\mathcal Y}}
\def\cZ{{\mathcal Z}}
\newcommand{\rmod}[1]{\: \mbox{mod} \: #1}

\def\cg{{\mathcal g}}

\def\vr{\mathbf r}

\def\e{{\mathbf{\,e}}}
\def\ep{{\mathbf{\,e}}_p}
\def\eq{{\mathbf{\,e}}_q}

\def\em{{\mathbf{\,e}}_m}

\def\Tr{{\mathrm{Tr}}}
\def\Nm{{\mathrm{Nm}}}

 \def\SS{{\mathbf{S}}}

\def\lcm{{\mathrm{lcm}}}

\def\({\left(}
\def\){\right)}
\def\l|{\left|}
\def\r|{\right|}
\def\fl#1{\left\lfloor#1\right\rfloor}
\def\rf#1{\left\lceil#1\right\rceil}
\def\flq#1{\langle #1 \rangle_q}

\def\mand{\qquad \mbox{and} \qquad}

\newcommand{\commA}[1]{\marginpar{%
\begin{color}{red}
\vskip-\baselineskip 
\raggedright\footnotesize
\itshape\hrule \smallskip A: #1\par\smallskip\hrule\end{color}}}

\newcommand{\commI}[1]{\marginpar{%
\begin{color}{magenta}
\vskip-\baselineskip 
\raggedright\footnotesize
\itshape\hrule \smallskip I: #1\par\smallskip\hrule\end{color}}}

\newcommand{\commJ}[1]{\marginpar{%
\begin{color}{blue}
\vskip-\baselineskip 
\raggedright\footnotesize
\itshape\hrule \smallskip J: #1\par\smallskip\hrule\end{color}}}




\hyphenation{re-pub-lished}

\mathsurround=1pt

\def\bfdefault{b}
\overfullrule=5pt

\def \F{{\mathbb F}}
\def \K{{\mathbb K}}
\def \Z{{\mathbb Z}}
\def \Q{{\mathbb Q}}
\def \R{{\mathbb R}}
\def \C{{\mathbb C}}
\def\Fp{\F_p}
\def \fp{\Fp^*}

\def\leg#1{\(\frac{#1}{p}\)}
\newcommand{\legendre}[2]{\genfrac{(}{)}{}{}{#1}{#2}}

\def\Smn{S_{k,\ell,q}(m,n)}

\def\Kmn{\cK_p(m,n)}
\def\psmn{\psi_p(m,n)}

\def\SM{\cS_{k,\ell,q}(\cM)}
\def\SMN{\cS_{k,\ell,q}(\cM,\cN)}
\def\SAMN{\cS_{k,\ell,q}(\cA;\cM,\cN)}
\def\SABMN{\cS_{k,\ell,q}(\cA,\cB;\cM,\cN)}

\def\SIJq{\cS_{k,\ell,q}(\cI,\cJ)}
\def\SAJq{\cS_{k,\ell,q}(\cA;\cJ)}
\def\SABJq{\cS_{k,\ell,q}(\cA, \cB;\cJ)}

\def\sM{\cS_{k,q}^*(\cM)}
\def\sMN{\cS_{k,q}^*(\cM,\cN)}
\def\sAMN{\cS_{k,q}^*(\cA;\cM,\cN)}
\def\sABMN{\cS_{k,q}^*(\cA,\cB;\cM,\cN)}

\def\sIJq{\cS_{k,q}^*(\cI,\cJ)}
\def\sAJq{\cS_{k,q}^*(\cA;\cJ)}
\def\sABJq{\cS_{k,q}^*(\cA, \cB;\cJ)}
\def\sABJp{\cS_{k,p}^*(\cA, \cB;\cJ)}

 \def \xbar{\overline x}

\title[Approximations to Fekete Polynomials]{Power Series Approximations to Fekete Polynomials}

\author[J. Bell]{Jason Bell }
\address{Department of Pure Mathematics, University of Waterloo,
Waterloo, Ontario, N2L 3G1, Canada}
\email{Jason Bell <jpbell@uwaterloo.ca>}

\author[I. E.~Shparlinski]{Igor E. Shparlinski}
\address{School of Mathematics and Statistics, University of New South Wales,
 Sydney, NSW 2052, Australia}
 
 \email{Igor E. Shparlinski <igor.shparlinski@unsw.edu.au>}

\begin{abstract} We study how well Fekete polynomials
$$
F_p(X) = \sum_{n=0}^{p-1} \leg{n} X^n \in \Z[X]
$$
with the coefficients given by Legendre symbols modulo a prime $p$, 
can be approximated by power series representing algebraic functions of 
a given degree. We also obtain some explicit results describing polynomial recurrence relations
which are satisfied by the coefficients of such algebraic functions. 
 \end{abstract}


\maketitle

\section{Introduction}
\label{sec:intro}

\subsection{Background and motivation}

For a prime $p\ge 3$ we recall that the {\it  Fekete polynomial\/}
$$
F_p(X) = \sum_{n=0}^{p-1} \leg{n} X^n \in \Z[X]
$$
is the polynomial of degree $p-1$
with coefficients given by Legendre symbols modulo  $p$.

Starting with the work of Conrey, Granville,   Poonen and Soundararajan~\cite{CGPS}, 
much attention has been devoted to analytic properties of these polynomials, such as 
the distribution of zeros and relations between various norms, see~\cite{Erd1,Erd2,ErdLub,GuSc} 
and references therein. 

Here we consider an apparently new question about the algebraic nature of Fekete polynomials.
Namely for an integer $N\ge 2$ we denote by $d_p(N)$ the smallest $d$ such that 
$$
F_p(X) \equiv G(X) \pmod {X^N}
$$
for some formal power series 
 \begin{equation}
\label{eq:GX}
G(X) = \sum_{n=0}^\infty A_n X^n \in \C[[X]]
\end{equation}
satisfying a nontrivial polynomial equation 
 \begin{equation}
\label{eq:hXGX}
h(X,G(X)) = 0
\end{equation}
for some polynomial $h(X,Y)\in \C[X,Y]$ of degree at most  $d$ in each variable. 
Clearly, we can always assume that $h$ is {\it irreducible\/}.

Just to show that this question has intrinsic number theoretic flavour, we note that 
until the Burgess~\cite{Burg} bound on the smallest quadratic nonresidue is improved
we cannot rule out that 
$$
d_p\(\fl{p^{1/4e^{1/2}}}\)=1, 
$$
corresponding to the function 
$$
G(X) = \frac{X}{1-X} =  \sum_{n=1}^\infty   X^n \in \Z[[X]].
$$
We also remark that the irrationality and transcendence of power series of multiplicative functions have been studied in a number of works, see~\cite{BaLuSh,BeBrCo,BoCo}.
However the question about the degree of approximating algebraic functions seems to be new. 

Here we obtain nontrivial lower bounds on  $d_p(N)$ starting from the values 
$N\ge p^{1/2}(\log p)^{1+\varepsilon}$ for any fixed $\varepsilon > 0$.

\subsection{Approach}

Our approach is based on first showing that the coefficients of powers series~\eqref{eq:GX} of algebraic 
functions satisfy a polynomial recurrence relation of the form
$$
\sum_{j=0}^L A_{n+j} P_j(n) = 0, \qquad n =0,1,\ldots, 
$$
with polynomials $P_j(T) \in \C[T]$ and give explicit bounds on the order $L$ of the relation 
and the largest degree $D$ of the polynomials $P_j$, $j=0, \ldots, L$, as the function of the degree
of $h$ in~\eqref{eq:hXGX}. 
We remark that, although in qualitative form, this is a known fact, we are unaware of explicit bounds on $D$ and
$L$ being given in the literature. So, here we fill this gap.

After this we also show that sequences that have enough oscillation 
cannot satisfy polynomial recurrences. 

Having a result of this type we then use character sums to show that 
the sequence of Legendre symbols satisfies this oscillatory property
on any interval of length  $N\ge p^{1/2+\varepsilon}$ with a fixed $\varepsilon > 0$.
This leads to a desired result. 

\subsection{General notation}

Throughout the paper,  as usual $A\ll B$  is equivalent to the inequality $|A|\le cB$
with some absolute constant $c>0$.

For a polynomial $f(X,Y)\in \C[X,Y]$ we use $\deg_X f$ and $\deg_Y f$ to
denote the degrees of $f$ with respect to $X$ and $Y$, respectively, 
reserving $\deg f$ for the total degree. 

Furthermore we also use $\cD_X$ and $\cD_T$ for partial differentiation operators 
 \begin{equation}
\label{eq:DiffOp}
\cD_X \Psi= \frac{\partial \Psi(X,T)}{\partial X} \mand 
\cD_T \Psi = \frac{\partial \Psi(X,T)}{\partial T}.
\end{equation}

\subsection{Degree of approximation}

Our main result is the following bound:

\begin{theorem}
\label{thm:Bound}  For any  sufficiently large prime  $p$ and positive integer  $N < p$ we have
$$
d_p(N)\gg  \(\frac{N}{p^{1/2}\log\, p}\)^{1/10}. 
$$
\end{theorem}

The proof rests on  Propositions~\ref{prop:PolyRec}  and~\ref{prop:HAT}, 
given in Sections~\ref{eq:poly rec} and~\ref{sec:appox},
respectively, and which are of a rather general
nature and thus can be
of independent interest. 

\section{Algebraic functions and  holonomic sequences}
\label{eq:poly rec}

\subsection{Explicit bounds on the order and degree of  polynomial recurrences}
As we have mentioned, qualitatively the fact that the coefficients of algebraic functions 
are {\it holonomic\/}---that is, they satisfy a recurrence relation with polynomial coefficients---is very well known. However, no quantitative form has been reported in the literature, so 
we now present such a result. 

\begin{prop}
\label{prop:PolyRec}  If a power series $G(X)$ of the form~\eqref{eq:GX} satisfies~\eqref{eq:hXGX} for
some irreducible 
 polynomial $h(X,Y)\in \C[X,Y]$  of total degree at most $d\ge 2$, 
then the coefficients of $G(X)$ satisfy a relation
$$
\sum_{j=0}^L A_{n+j} P_j(n) = 0, \qquad n =0,1,\ldots, 
$$
with polynomials $P_j(T) \in \C[T]$ with a nonzero polynomial $P_L(T)$ 
and such that 
$$
L \le 4d^2 \mand \deg_T P_j \le 3(d+1)^2, \quad j =0, \ldots, L.
$$
\end{prop}


\subsection{Differential equation for algebraic functions}
We first recall a result of Bostan, Chyzak, Lecerf, Salvy and Schost~\cite[Theorem~2]{BCLSS}
which descripes the shape of differential equations for algebraic functions.

\begin{lem}
\label{lem:DegLength}
Let $K$ be a field of characteristic zero.   If a power serie
$$
G(X) = \sum_{n=0}^\infty A_n X^n \in K[[X]]
$$
satisfies~\eqref{eq:hXGX} for
some irreducible 
 polynomial $h(X,Y)\in \C[X,Y]$  of total degree at most $d\ge 2$, 
then   $G(X)$ satisfies a differential equation of the form
$$
\sum_{i=0}^N Q_i(X) G^{(i)}(X) = 0
$$
of order $N \le 6\deg_Y h$ with polynomials $Q_i(X) \in K[X]$,  
$i =0, \ldots, N$,  of degree at most $D$, where 
$$D \le 3\deg_X h \deg_Y h 
$$
with $Q_N$ nonzero. 
\end{lem}

\subsection{Proof of Proposition~\ref{prop:PolyRec}}

The differential equation of Lemma~\ref{lem:DegLength},   gives us a polynomial recurrence 
for  the coefficients $A_n$ of $T^n$ in $G(T)$ as in~\eqref{eq:GX}. More precisely, we have
$$
\sum_{j=0}^{N+D} P_j(n) A_{n+i}=0, \qquad n = 0, 1, \ldots 
$$
 with polynomials $P_j(T) \in \C[T]$, $j =0, \ldots, N+D$,  of degree at most $D$.
 We have $D \le 3d^2$, $N \le 6d$ and thus $N + D \le 3(d+1)^2$.   The result now follows.


\section{Approximation by holonomic sequences}
\label{sec:appox}

\subsection{Holonomic approximation}

We are now able to establish our main technical result. As we have mentioned 
we present it in a general form, which makes it suitable for applications to various arithmetic functions. 

As usual, for a complex $z$, we use $\overline z$ to denote its complex conjugate.

\begin{prop}\label{prop:HAT}  
Let $m$ and $D$ be positive integers and let $\kappa$ and $\tau$ be positive real numbers with
 $4\tau \ge \kappa$, let $Q_1(X),\ldots ,Q_m(X)$ be complex polynomials of degree at most $D$, and let $f:\Z\to \C$.  
Suppose that $f(n)$ satisfies:
\begin{enumerate}
\item[(i)] $|f(a)|^2+\cdots + |f(b)|^2 \ge \kappa(b-a)$ for $b>a$;
\item[(ii)] $\left|\sum_{n=a}^b f(n+j) \overline{f(n+h)}\right| \le \tau$
for all integers $a<b$ and $j<h \le m$. 
\end{enumerate}
Then for $$
\Delta(n) = \sum_{j=1}^m Q_j(n)f(n+j)$$ 
we have $\Delta(n) \ne 0$ for some natural number  $n<168 \tau \kappa^{-1} e^3 D^3 m^4$. 
\end{prop}

\subsection{General definitions and notation}

Given a polynomial $Q(X)\in \C[X]$ we create a subset of $\C$, $C(Q)$, which is defined as follows.  We write $Q(X)=A(X)+iB(X)$, 
where $A$ and $B$ are polynomials with real coefficients.  Then:
\begin{itemize}
\item if $A(X)$ and $B(X)$ are both nonconstant then we define $C(Q)$ to be the set of the zeros of $A'(X)B'(X)$;
\item  if $A(X)$ is constant and $B$ is nonconstant then $C(Q)$ is the set of zeros of $B'(X)$; if $A(X)$ is 
nonconstant and $B$ is constant then $C(Q)$ is the set of zeros of $A'(X)$; \item if $A$ and $B$ are constant then $C(Q)$ is the empty set.
\end{itemize}

\subsection{Bounds of some sums with polynomials} 

\begin{lem} \label{lem: 1} Let $D$ be a positive integer, let $\kappa>0$, let $Q(X)$ be a nonzero polynomial of degree at most $D$, and let $f:\Z\to \C$.  Suppose that $a$ and $b$ are natural numbers with $a<b$ and $L$ is a positive integer such that:
\begin{enumerate}
\item[(i)]  $|j-\lambda|\ge a/LD$ for $j\in \{a,a+1,\ldots ,b\}$ and for all roots $\lambda$ of $Q$;
\item[(ii)]  $(b-a)/a\le 1/LD^2$;
\item[(iii)]  $|f(a)|^2+\cdots +|f(b)|^2 \ge \kappa (b-a)$.
\end{enumerate}
Then 
$$
\sum_{j=a}^b |Q(j)|^2|f(j)|^2\ge \frac{(b-a)\kappa}{e^2} \max_{a\le j\le b} |Q(j)|^2.$$
\end{lem}

\begin{proof}
Let $\lambda_1,\ldots ,\lambda_r$ denote the roots of $Q(X)$ with multiplicity, with $r\le D$.  Then let $b_0$ be a value of $j\in \{a,\ldots ,b\}$ where the maximum of $|Q(j)|$ is attained and let $a_0$ be a value of $j$ for $j\in \{a,\ldots ,b\}$ where the minimum of $|Q(j)|$ is attained.
$$Q(b_0)/Q(a_0)=\prod_{j=1}^r (b_0-\lambda_j)/(a_0-\lambda_j) = \prod_{j=1}^r \left( 1+(b_0-a_0)/(a_0-\lambda_j) \right).$$
Now by assumption $ |a_0-\lambda_j| \ge a/LD$ for $j=1,\ldots ,r$ and so we obtain the inequality
$$|Q(b_0)|/|Q(a_0)| \le \prod_{j=1}^r \left( 1 + LD(b-a)/a\right).$$
Since $(b-a)/a\le 1/LD^2$, we now see that
 \begin{equation}
\label{eq:MinMaxQ}
 |Q(b_0)|/|Q(a_0)| \le (1+1/D)^r \le (1+1/D)^D \le e.
\end{equation}
We also  see that
\begin{align*}
 |Q(a)|^2|f(a)|^2 &+|Q(a+1)|^2 |f(a+1)|^2+\cdots + |Q(b)|^2 |f(b)|^2 \\ &\ge  |Q(a_0)|^2 (|f(a)|^2+\cdots +|f(b)|^2)  \ge
\kappa (b-a)  |Q(a_0)|^2.
\end{align*}
Thus, recalling~\eqref{eq:MinMaxQ}, we obtain the result. 
\end{proof}

\begin{lem}
Let $Q(X)$ be a complex polynomial, let $f:\Z\to \C$, and let $a$ and $b$ be nonnegative integers such that $[a,b]\cap C(Q)$ is empty. 
Suppose that $\tau>0$ is such that for every integer $n$ and for every $r\in \{a,a+1,\ldots ,b\}$ we have
$$
\left|\sum_{j=a}^r f(n+j)\right| \le \tau.
$$  
Then
$$
\left|\sum_{j=a}^b Q(j) f(n+j)\right| \le 4\tau(|Q(b)|+|Q(a)|).$$  
\label{lem: 2}
\end{lem}
\begin{proof} Write $Q(X)=A(X)+iB(X)$ with $A$ and $B$ real polynomials.
Since $C\cap [a,b]=\emptyset$, we have that $A(X)$ and $B(X)$ are monotonic on the interval $[a,b]$. Then using summation by parts we see that 
\begin{align*} 
& \left| \sum_{j=a}^b Q(i) f(n+j)\right|\\
 & \qquad \qquad =\left| Q(b) \sum_{j=a}^b f(n+j) + \sum_{r=a}^{b-1} \(Q(r)-Q(r+1)\) \sum_{j=a}^r f(n+j)\right| \\
 & \qquad \qquad  \le \tau  \left(|Q(b)| + \sum_{r=a}^{b-1} |Q(r)-Q(r+1)|\right).
 \end{align*}
Thus to complete the proof, it is enough to show that 
$$|Q(b)| + \sum_{r=a}^{b-1} |Q(r)-Q(r+1)|\le 4|Q(b)|+4|Q(a)|.$$
Notice that
\begin{align*} 
  |Q(b)| & + \sum_{r=a}^{b-1} |Q(r)-Q(r+1)|\\
& \le    |A(b)|+|B(b)|\\
& \qquad +\sum_{r=a}^{b-1} \(|A(r)-A(r+1)|+|B(r)-B(r+1)|\).
\end{align*}
Since $C(Q)\cap [a,b] =\emptyset$, we see that $A$ and $B$ are monotonic on $[a,b]$ and thus 
$$
\sum_{r=a}^{b-1} (|A(r)-A(r+1)| +|B(r)-B(r+1)|) = |A(b)-A(a)| + |B(b)-B(a)|.$$
In turn this implies that 
\begin{align*}
|Q(b)| + \sum_{r=a}^{b-1}  |Q&(r)-Q(r+1)|\\ & \le   2|A(b)|+2|B(b)| + |A(a)|+|B(a)|\\
& \le   4|Q(b)|+2|Q(a)|  \le   4|Q(b)|+4|Q(a)|,\end{align*}
 as required.
\end{proof}

\begin{lem}
\label{lem: 3}
Let $m$ and $D$ be positive integers and let $\kappa$ and $\tau$ be positive real numbers, let $Q_1(X),\ldots ,Q_m(X)$ be complex polynomials of degree at most $D$, and let $f:\Z\to \C$.  Suppose that $a$ and $b$ and $L$ are positive integers that have the following properties:
\begin{enumerate} 
\item[(i)] $b-a>4\tau(m-1)e^2/\kappa$;
\item[(ii)] $|j-\lambda|\ge a/LD$ for $j\in \{a,a+1,\ldots ,b\}$ and for all roots $\lambda$ of $Q$;
\item[(iii)] $(b-a)/a\le 1/D^2L$;
\item[(iv)] $(\cup_{i<j} C(Q_i Q_j))\cap [a,b]=\emptyset$. 
\end{enumerate}
Suppose in addition that $f(n)$ satisfies:
\begin{enumerate}
\item[(I)] $|f(a+j)|^2+\cdots + |f(b+j)|^2 \ge \kappa(b-a)$ for $j=1,\ldots , m$;
\item[(II)] $\left|\sum_{n=a}^r f(n+j)f(n+k)\right| \le \tau$ for all integers $n$ and all $r\in \{a,\ldots ,b\}$.
\end{enumerate}
Then for $$
\Delta(n) = \sum_{j=1}^m Q_j(n)f(n+j)$$ 
we have $\Delta(n) \ne 0$ for some natural number  $n\in \{a,a+1,\ldots ,b\}$. 
\end{lem}

\begin{proof}
It suffices to show that
$$\sum_{n=a}^b \left|\sum_{j=1}^m Q_j(n)f(n+j)\right|^2 > 0. $$  
Expanding, we see this is greater than or equal to 
\begin{align*}
\sum_{n=a}^b & \left|\sum_{j=1}^m Q_j(n)f(n+j)\right|^2 \\ 
&\qquad = 
\sum_{n=a}^b \sum_{j=1}^m |Q_j(n)|^2 |f(n+j)|^2 \\
& \qquad \qquad  -\sum_{1\le j<k\le m}\left| \sum_{n=a}^b  Q_j(n)\overline{Q_k(n)} f(n+j)\overline{f(n+k)}\right|.
\end{align*}
Since 
$$
\left|\sum_{n=a}^r f(n+j)\overline{f(n+k)}\right| \le \tau,
$$ 
we then see by Lemma~\ref{lem: 2} that  to
\begin{equation}
\begin{split}
\label{eq: bound1}\sum_{n=a}^b & \left|\sum_{j=1}^m Q_j(n)f(n+j)\right|^2 \\ 
&\qquad \ge \sum_{j=1}^m \sum_{n=a}^b |Q_j(n)|^2 |f(n+j)|^2 \\
& \qquad\qquad\qquad -4\tau \sum_{1\le j<k\le m} (|Q_j(a)Q_k(a)|+ |Q_j(b)Q_k(b)|).
\end{split}
\end{equation}
Let $M_j= \max\left\{|Q_j(a)|^2,|Q_j(b)|^2\right\}$, $j = 1, \ldots, m$. 
By Lemma~\ref{lem: 1}, we see that
$$
\sum_{j=1}^m \sum_{n=a}^b |Q_j(n)|^2 |f(n+j)|^2 \ge \frac{(b-a)\kappa}{e^2} \sum_{j=1}^m M_j,
$$ while we have
\begin{align*} |Q_j(a)Q_k(a)| &+ |Q_j(b)Q_k(b)|\\
& \le  
\frac{1}{2}\left(|Q_j(a)|^2 +|Q_k(a)|^2 + |Q_j(b)|^2+|Q_k(b)|^2\right) \\
&\le M_j+M_k.
\end{align*}
Then we see that the quantity on the right-hand side of~\eqref{eq: bound1} is greater than or equal to 
\begin{align*} 
\sum_{j=1}^m \frac{(b-a)\kappa}{e^2} M_j &- \sum_{1\le j<k\le m} 4\tau(M_j+M_k) \\
&=  \sum_{j-1}^m \left(\frac{(b-a)\kappa}{e^2} -\tau(m-1)\right)M_j>0.
\end{align*}
Since $\kappa(b-a)>4\tau(m-1)e^2$ and  by~(ii), we have $M_i>0$, $j = 1, \ldots, m$, we obtain the desired result.
\end{proof}

\begin{lem}
Let $Q_1(X),\ldots ,Q_m(X)$ be polynomials of degree at most $D$.  Given $A\ge 1$ there exist
positive integers $a$ and $b$ and $L$ with 
$$ 
b<20 A e^3 D^2 m^3 (A+Dm) \mand L= 9 m^2 
$$
satisfying:
\begin{enumerate} 
\item[(i)] $b-a>A(m-1)e^2$;
\item[(ii)] $|j-\lambda|\ge a/LD$ for $j\in \{a,a+1,\ldots ,b\}$ and for all roots $\lambda$ of each $Q_i$ that is nonzero;
\item[(iii)] $(b-a)/a\le 1/D^2L$;
\item[(iv)] $(\cup_{i<j} C(Q_i Q_j) \cap [a,b]=\emptyset$.
\end{enumerate}
\label{lem: 4}
\end{lem}
\begin{proof}
The inequalities $b-a>A(m-1)e^2$ and $b-a\le a/D^2L$ are satisfied whenever $a\ge AD^2Lm^2(m-1)e^2$ and $b\le a(1+1/D^2L)$.  
For an integer $R>A(m-1)e^2$, consider the interval $[D^2LR,D^2LR+R]$. We show that we can find some $R$ such that we can take $a=D^2LR$ and $b=2D^2LR+R$ and~(i)--(iv) hold.  

Notice that if there is some $\lambda$ that is a root of a nonzero $Q_i$ such that
$|j-\lambda|<DR$ for some $j\in [D^2LR,D^2LR+R]$, then for the real part of $\lambda$ we have:
$$
\Re \lambda \in [D^2LR-DR,D^2LR+DR+R].
$$ 
 In particular, since there are at most $Dm$ values of $\lambda$ we see that if $t>Dm$ and $R_1,\ldots ,R_t$ are integers with each $R_i>A(m-1)e^2$ such that the intervals $[D^2L R_i-DR_i,D^2LR_i+DR_i+R_i]$ are pairwise disjoint then there are at least $t-Dm$ values of $i$ in $\{1,\ldots ,t\}$ such that $|j-\lambda|<DR_i$ for $j\in [D^2LR_i,D^2LR_i+R_i]$ and all roots of nonzero $Q_i$.  Now we define
$R_1=\lceil A(m-1)e^2+1 \rceil$, and then for $i>1$ we define $R_i$ recursively as 
$$
R_i = \lceil R_{i-1} (D^2L+D+1)/(D^2L-D) \rceil.
$$ 
Then we see that $D^2 L R_{i+1}-DR_{i+1} > D^2L  R_i +DR_i+R_i$ and so the intervals  
$[D^2L R_i-DR_i,D^2LR_i+DR_i+R_i]$ are pairwise disjoint.  By induction, we have
$$
R_i < \left( \frac{D^2L+D+1}{D^2L-D}\right)^{i-1}(A(m-1)e^2 +i),
$$ and so if we take 
$t=Dm+2Dm^2$ then there are at least $2Dm^2$ values of $R_i$ with $i\le t$ such that 
 $|j-\lambda|<DR_i$ for $j\in [D^2LR_i,D^2LR_i+R_i]$ and all roots $\lambda$ of nonzero $Q_i$.  Now since $$\bigcup_{i<j} C(Q_i Q_j)$$ has size at most $2Dm^2$, we then see that there is some $i$ with $i\le t$ such that
$a=2D^2mR_i$ and $b=2D^2mR_i+R_i$ satisfy conditions (i)--(iv). 

Recall that   $L=9m^2$. Then for $D \ge 2$ and any $m\ge 1$ we have the following 
obvious inequalities
\begin{equation}
\begin{split}
\label{eq: Elem ineq}
\(1+\frac{2D+1}{D^2L-D}\)^{Dm+2Dm^2}&<\(1+\frac{3D}{9D^2m^2}\)^{Dm+2Dm^2}\\
& < \(1+\frac{1}{3Dm^2}\)^{3Dm^2} <e.
\end{split}
\end{equation}
For $D=1$, using elementary calculus one also derives that 
$$
\(1+\frac{3}{9m^2-1}\)^{m+2m^2}  <e, 
$$
for any $m \ge 1$ thus~\eqref{eq: Elem ineq} also holds in this case.

Therefore, using~\eqref{eq: Elem ineq}, we derive
\begin{align*}
b& \le   (2D^2 L+1)R_t \\
& \le   (2D^2L+1) \left( \frac{D^2L+D+1}{2D^2L-D}\right)^{Dm+2Dm^2}\\
& \qquad \qquad \qquad  \qquad \cdot (A(m-1)e^2 +Dm+2Dm^2)\\
& \le  (2D^2L+1) (1+(2D+1)/(D^2L-D))^{Dm+2Dm^2}\\
& \qquad  \qquad \qquad  \qquad \cdot(A(m-1)e^2 +Dm+2Dm^2)\\
&\le   20 eD^2 m^2  \cdot (Ame^2+3Dm^2) \le 20 e^3 D^2 m^3 (A+Dm),
\end{align*}
which concludes the proof.
\end{proof}

\subsection{Concluding the proof of Proposition~\ref{prop:HAT}}
The result follows immediately from Lemmas~\ref{lem: 3} and~\ref{lem: 4} 
taking $A=4\tau/\kappa$.

\section{Proof of Theorem~\ref{thm:Bound}}

%
%
%
%
%
%
%
%
%

\subsection{Incomplete character sums}

We recall the following well-known bound on incomplete character sums which
 follows from the Weil bound on complete mixed sums of multiplicative and 
additive characters and the standard reduction between complete  and  incomplete sums,
see~\cite[Section~12.2]{IwKow}.

  \begin{lem}
\label{lem:IncompSum}  
Let $p$ be an odd prime, let $f(n)=\leg{n}$, 
and let $g(n)=f(n+j) f(n+h)$ for some natural numbers   $j<h<p$.  
Then for any natural number $K< p$ we have 
$$
 \left| \sum_{k=1}^K g(n+k)\right|\le B  p^{1/2} \log p
$$
for some absolute constant $B$. 
\end{lem}

\subsection{Concluding the proof}
We note that 
$$
\sum_{n=a}^b |f(n)|^2 \ge (b-a)
$$ for $a<b$ and by 
Lemma~\ref{lem:IncompSum} 
we   can apply Proposition~\ref{prop:HAT} to $f(n)$ with 
$$\kappa=1 \mand \tau \ll p^{1/2}\log\, p
$$ 
and we see that if $G(X)$ is a holonomic power series whose sequence of coefficients satisfy a polynomial recurrence of length at most $m$ and of degree at most $D$ then $G(x)$ can agree with the Fekete polynomial $F_p(X)$  
up to degree $N$ at most 
$$
N \ll D^2 m^3 (\tau + Dm) \ll D^2 m^3 \(p^{1/2}\log\, p + Dm\) .
$$
In particular, if $G(X)$ is an algebraic power series satisfying a polynomial equation $h(X,G(X))=0$, where $h(X,Y)$ has total degree at most $d_p(N)$, then by Proposition~\ref{prop:PolyRec} 
this gives that we can take 
$m\le 4d_p(N)^2$ and $D\le 3(d_p(N)+1)^2$ and so
 \begin{equation}
\label{eq:Bound}
N\ll  d_p(N)^{10}  \(p^{1/2}\log\, p + d_p(N)^{4}\). 
\end{equation} 
Clearly, if $d_p(N)^{4} \ge p^{1/2}\log\, p$ then 
$d_p(N) \ge p^{1/8} \ge N^{1/8}$ and thus there is nothing to prove.
Otherwise we  derive from~\eqref{eq:Bound} that 
$N\ll  d_p(N)^{10}  p^{1/2}\log\, p$ 
and the result follows.


\section{Comments}

We note that it is not difficult (but somewhat tedious) to get an explicit version 
 of Theorem~\ref{thm:Bound}. Furthermore,  a slight modification of the proof of Theorem~\ref{thm:Bound}
allows one to remove the factor $\log\, p$ in the denominator of its lower bound.
However, a much more challenging question is to obtain a nontrivial lower 
bound on $d_p(N)$  for $N$ below the $p^{1/2}$-threshold. 
Within our approach, 
this rests on the existence of nontrivial bounds on short character  sums with linear and quadratic 
polynomials. For linear polynomials the Burgess bound (see~\cite[Theorem~12.6]{IwKow})
provides such a necessary tool. However, for quadratic polynomials the problem  obtaining 
nontrivial bounds for sums of length below $p^{1/2}$ is still widely open, 
see the survey~\cite{Chang}.

\section*{Acknowledgement}

The authors are very grateful to Alin Bostan and Arne Winterhof for their interest and 
many useful suggestions. In particular, Alin Bostan h  informed us 
about~\cite[Theorem~2]{BCLSS}, which allowed us to improve our initial 
results. 

The second author gratefully acknowledges the support, the hospitality
and the excellent conditions at RICAM, Austrian Academy of Science, Linz, 
during his visit.

The first author was supported by NSERC grant RGPIN-2016-03632; 
the second author was  supported   by ARC Grant~DP140100118.

\end{document}